\documentclass[12pt,centertags,oneside]{amsart}

\usepackage[utf8]{inputenc}  
\usepackage[T1]{fontenc} 

\usepackage[backend=biber, style=alphabetic]{biblatex}
\renewbibmacro{in:}{}

\usepackage{typearea}
\usepackage[hidelinks]{hyperref}
\usepackage{stmaryrd}
\usepackage{amssymb}
\usepackage{a4wide}
\usepackage[mathscr]{eucal}
\usepackage{mathrsfs}
\usepackage{dsfont} 
\usepackage{enumitem} 
\usepackage{typearea}
\usepackage{charter}
\usepackage{pdfsync}
\usepackage[a4paper,width=16.2cm,top=3cm,bottom=3cm]{geometry}
\usepackage{amsmath}
\usepackage{amstext}
\usepackage{amsthm}
\usepackage{amscd}

\usepackage{color}
\usepackage{xcolor}

\numberwithin{equation}{section}

\newtheorem{theorem}{theorem}[section]
\newtheorem{theo}[theorem]{Theorem}
\newtheorem{proposition}[theorem]{Proposition}

\newtheorem{lemma}[theorem]{Lemma}

\theoremstyle{definition}
\newtheorem{definition}[theorem]{Definition}

\DeclareMathOperator{\Stab}{\mu - Stab}
\DeclareMathOperator{\stab}{Stab}

\newcommand{\cali}[1]{\mathscr{#1}}

\newcommand{\supp}{{\rm supp}}

\newcommand{\Jac}{{\rm Jac}}

\newcommand{\vol}{{\rm vol}}
\newcommand{\luc}{\textrm{luc}}

\newcommand{\C}{\mathbb{C}}

\newcommand{\N}{\mathbb{N}}

\renewcommand{\P}{\mathbb{P}}

\renewcommand{\epsilon}{\varepsilon}

\title[Ramification current, post-critical normality and  stability.]
{Ramification current, post-critical normality and  stability of holomorphic endomorphisms of $\P^k$}

\author{Fran\c cois Berteloot and Maxence Brevard}
\address{Universit\'e Toulouse 3, 
Institut Math\'ematique de Toulouse, 
118 route de Narbonne,
F-31062 Toulouse Cedex 9, France. }
\email{francois.berteloot@math.univ-toulouse.fr}
\email{maxence.brevard@math.univ-toulouse.fr}

\date{}

\begin{document}

\begin{abstract}
In the context of holomorphic families of endomorphisms of $\P^k$,
we prove that stability in the sense of \cite{BBD} is equivalent to
a summability condition for the post-critical mass and to the convergence of  a suitably defined ramification current.
This allows us to both simplify the approach of \cite{BBD}   and better relate stability to post-critical normality.
\end{abstract}

\maketitle

\medskip\medskip

 \medskip

\noindent



\section{Introduction and results}

A \emph{holomorphic family of degree $d\ge 2$ endomorphisms on $\P^k$}, parametrized by a complex manifold $M$, is a holomorphic map $f:M\times \P^k\to M\times \P^k$, of the form $(\lambda,z)\mapsto (\lambda,f_\lambda(z))$ such that, as endomorphisms on $\P^k$, the maps $f_\lambda$ have the same algebraic degree
$d\ge 2$.\\
In dimension $k=1$,  the dynamical stability of such families has been independently studied by
Ma\~{n}\'e-Sad-Sullivan \cite{MSS} and Lyubich \cite{LYU}.
At the heart of their theory stands  a characterization  of global dynamical stability by a post-critical normality condition encoding the dynamical stability of  the critical set.
By considering a cover of the parameter space, one may always assume that
the critical points are marked, that is to say that the critical set
$C_f$ of $f$ is given by the graphs of $2d-2$ holomorphic maps $c_j : M\to \P^1$.
Under this mild assumption, this characterization can be stated as follows.

\begin{theo}\label{TheoMSSL}Let $f:M\times \P^1\to M\times \P^1$ be a holomorphic family of rational functions
with marked critical points $c_j$ and let $J_\lambda$ denote the Julia set of $f_\lambda$.
    Then the two following assertions are equivalent:
\begin{itemize}
\item[i)] the Julia sets $J_\lambda$ move holomorphically with $\lambda$,
\item[ii)] the sequences $(f^{\circ n} \circ c_j)_n$ are normal.
\end{itemize}
\end{theo}


The fact that the Julia sets $J_\lambda$ move holomorphically with $\lambda$ amounts to say that there exists a holomorphic lamination
$\cali L$ by graphs over $M$ in $M\times \P^1$ whose slices ${\cali L}_\lambda:={\cali L} \cap (\{\lambda\}\times \P^1)$ coincide with $J_\lambda$ for every $\lambda \in M$;
in  that case the family $f$ is said to be \emph{stable}.
More generally, the  \emph{stability locus} $\stab(f)$ of a family $f$ is the set of parameters $\lambda \in M$ admitting
a neighbourhood $D$ such that the  restricted family $f\vert_{D\times \P}$ is stable.
An important corollary of the above theorem is the density of the stability locus in the parameter space.
Another one, due to DeMarco \cite{DM}, is that the Lyapunov exponents $L(\lambda)$ of the maximal entropy measures of $f_\lambda$ define a plurisubharmonic function on $M$ which is pluriharmonic exactly on $\stab(f)$.\\

When studying  stability of families of endomorphisms of $\P^{k\ge 2}$,
basic one-dimensional tools are no longer efficient, and a  different  approach,
mainly  based on pluripotential and ergodic techniques, has recently  been 
designed  in \cite{BBD}. To describe it, we first briefly recall some basic facts and definitions.
For any parameter $\lambda \in M$, denote by $\mu_\lambda$ the equilibrium measure of $f_\lambda$ and by $J_\lambda := \supp(\mu_\lambda)$ its Julia set.
The space $\mathcal O(M,\P^k)$ of holomorphic maps from $M$ to $\P^k$  is endowed with the topology of local uniform convergence;
which yields a complete metric space
$(\mathcal O(M,\P^k),d_{\luc})$ on which the family $f$ induces a topological dynamical system  $(\mathcal O(M,\P^k),\cali F)$:
$$\forall \gamma \in \mathcal O(M,\P^k), \forall \lambda \in M \::\: \cali F(\gamma)(\lambda) := (f_\lambda \circ \gamma)(\lambda).$$
The following (possibly empty) $\cali F$-invariant subset of $ \mathcal O(M,\P^k)$ is also of importance
$$\cali J := \{\gamma \in \mathcal O(M,\P^k) \: : \: \gamma(\lambda) \in J_\lambda \text{ for every } \lambda \in M\}.$$

A  key idea is to replace a holomorphic motion of Julia sets by a holomorphic motion of
the equilibrium measures, meaning by that the existence of a holomorphic lamination $\cali L$ by graphs over $M$ in $M\times \P^k$ whose
slices ${\cali L}_\lambda:={\cali L} \cap (\{\lambda\}\times \P^k)$ have full $\mu_\lambda$ measure for every $\lambda \in M$.
Such a lamination can be further required to be induced by a  relatively compact subset of $\cali J$
on which $\cali F$ induces a $d^k$ to $1$ map and, in that case, is called an \emph{equilibrium lamination}.
This leads to  the following definition  of stability.  
    
\begin{definition}
A family $f$ is said to be \emph{$\mu$-stable} if it admits an  equilibrium lamination. The $\mu$-\emph{stability locus}  of a family $f$, denoted $\Stab(f)$, is the set of parameters $\lambda \in M$ for which there exists a neighbourhood $D$ in $M$ such that  the restricted family $f\vert_{D\times \P^k}$ is $\mu$-stable.
\end{definition}

 Although some post-critical normality conditions  played a crucial role in the work \cite{BBD}, 
they were not directly related to stability. We aim here to both introduce a new notion of  post-critical normality and show that it
naturally implies $\mu$-stability. 
As we shall explain in section \ref{Section2}, this allows to simplify the overall scheme of proof  of \cite{BBD}. To this purpose, we borrow the concept of \emph{ramification current} from the work
{\cite{DS2}, see also \cite{DS}}, of Dinh and Sibony on the equidistribution of iterated preimages
for endomorphisms on $\P^k$ and polynomial-like maps, 
and adapt it to the setting of holomorphic families.

\begin{definition}
Let $C_f$ denote the critical set of $f$. For every integer $n \ge 0$, we denote by $R_{n,f}$ the closed positive current on $M\times \P^k$ defined
by 
$$R_{n,f}:=d^{-kn} (f^{\circ n})_* [ f(C_f)]$$
where $[f(C_f)]$ denotes the current of integration on the analytic subset $f(C_f)$. 
The \emph{formal ramification current} $R_f$  of $f$ is defined by
$$R_f :=\sum_{n\ge 0} R_{n,f}=\sum_{n\ge 0} d^{-kn} (f^{\circ n})_* [ f(C_f)].$$
The \emph{convergence domain} $\Omega(R_f)$ of $R_f$ is the  subset of points in $M\times\P^k$
which admit a neighbourhood $U$ such that the series $\sum_{n\ge 0} \Vert \mathds{1}_U R_{n,f}\Vert$ converges.
\end{definition}

Let us stress that $R_f$ induces a closed positive $(1,1)$-current on its convergence domain $\Omega(R_f)$.\\

We can now state our main result.

\begin{theo}\label{MainTheo} Let $f:M\times \P^k\to M\times \P^k$  be a holomorphic family of endomorphisms
of degree $d\ge 2$ on $\P^k$ and $\pi_M : M\times \P^k \to M$ be  the canonical projection.
Then: $$\Omega(R_f)=\pi_M^{-1} ({\Stab}(f)) =\pi_M (\Omega(R_f))\times \P^k.$$
\end{theo}

As we shall see in the next section, it follows from Theorem \ref{MainTheo} that $\lambda_0\in {\Stab}(f)$ if and only if
$\Vert  (f^{\circ n})_* [C_f]\Vert_{U\times \P^k}= O(d^{(k-1)n})$ for some  neighbourhood $U$ of $\lambda_0$ in $M$.
This last condition already appeared in \cite[Proposition 3.12(3)]{BBD} and, since for every $\lambda$,
$\Vert{  (f_{\lambda}^{\circ n})}_* [C_{f_\lambda}]\Vert_{ \P^k} \sim_n d^{(k-1)n}$,
it can be interpreted as a post-critical normality statement.
In dimension $k=1$, one may indeed check that it is equivalent to the second assertion  of Theorem \ref{TheoMSSL}
 (see \cite[Theorem 3.9, proof of $(5) \implies (6)$]{BB1}).\\

\noindent {\bf Notations:}
The mass of a positive current $T$ on some Borel measurable set $W$ is denoted $\Vert T\Vert_W$ or $\lVert \mathds{1}_W T \rVert$.
The set of all holomorphic maps from $U$ to $V$ is denoted ${\cali O}(U,V)$.
When $V=\C$, we simply note $\cali O(U)$. The euclidean 
ball of radius $r$ and centered at point $a$ in $\C^m$ is denoted $B(a,r)$.
The determinant of the jacobian matrix of an endormorphism $f : \P^k \to \P^k$ (respectively a polynomial map $F:\C^k \to \C^k$)
will be denoted $\Jac f$ (resp. $\Jac F$). The cardinality of a finite set $A$ is denoted $|A|$.

\section{The role of theorem \ref{MainTheo} in the stability theory}\label{Section2}

From now on,  $f$ is a holomorphic family of degree $d\ge 2$ endomorphisms on $\P^k$, parametrized by a complex manifold $M$ of complex dimension $m$.
We recall that the equilibrium measure $\mu_\lambda$ of $f_\lambda$ is a Monge-Amp\`ere mass given by
$\mu_\lambda=(dd^c_{z} g(\lambda,z) +\omega_{\P^k})^k$,
where $g$ is a Green function (see \cite[Theorem 1.16]{DS}), which is continuous on $M\times \P^k$.
We denote by $L(\lambda)=\int_{\P^k} \ln \vert \Jac f_\lambda(z)\vert\; d\mu_\lambda(z)$
the sum of Lyapunov exponents of $(f_\lambda,\mu_\lambda)$.
For every integer $n$ and every parameter $\lambda$, we denote by  ${\cali R}_n(\lambda)$
the set of $n$-periodic repelling points of $f_\lambda$ which belong to the Julia set of $f_\lambda$. We now introduce some weaker notions of stability.

\begin{definition}
A family $f$ is said to be \emph{weakly stable} (resp. \emph{asymptotically weakly stable}) if,
for every compact subset $M_0 \Subset M$, there exists a sequence of holomorphic laminations
$({\cali L}_n)_\N$ by graphs over $M$ in $M_0\times \P^k$ such that
$({\cali L}_n)_\lambda:= {\cali L}_n \cap (\{\lambda\}\times \P^k)= {\cali R}_n(\lambda)$ 
(resp.  $({\cali L}_n)_\lambda \subset {\cali R}_n(\lambda)$ and $|({\cali L}_n)_\lambda|  \sim_n
|{\cali R}_n(\lambda)|$) for every $\lambda \in M_0$.
\end{definition}

The following result combines those of \cite{BBD} with Theorem \ref{MainTheo}.
It both offers a much direct proof of some results of \cite{BBD} and completes them with a sharper characterization of stability in terms of post-critical normality.

\begin{theo}\label{theoStab}
Let $f:M\times \P^k\to M\times \P^k$  be a holomorphic family of endomorphisms of degree $d\ge 2$
on $\P^k$ parametrized by a simply connected complex manifold $M$. Then the following assertions are equivalent:
\begin{itemize}
\item[1)] $L$ is pluriharmonic on $M$,
\item[2)] $\Vert f^{\circ n}_* [C_f]\Vert_{U\times \P^k}= O(d^{(k-1)n})$
for every $U \Subset M$,
\item[3)] $\pi_M(\Omega(R_f))=M$, 
\item[4)] $f$ is {$\mu$}-stable,
\item[5)] $f$ is asymptotically weakly stable.
\end{itemize}
\end{theo}

Let us stress that in \cite{BBD} the assertions $1$) and $2$) were shown to be equivalent to weak stability
in the family of all degree $d\ge 2$ holomorphic endomorphisms of $\P^k$ or for any family
when $k=2$, while in \cite{Bia}, using further techniques, Bianchi proved the equivalence of assertions
$1$), $4$), $5$) and a version of $2$) in the larger setting of  holomorphic families of polynomial maps of large topological degree.\\

We now recall some fundamental tools from  \cite{BBD} which will also  be useful  here.
Our main thread is that the stability  properties are encoded in the Lyapunov function $L(\lambda)$,
from which they can be extracted by mean of the two following formulas:\\
\begin{itemize}[label={}, leftmargin=*]
    \item  $\boldsymbol{dd^c L}$\textbf{-formula}:
    $dd^cL=\pi_{M*}\left((dd^c_{\lambda,z} g(\lambda,z) +\omega_{\P^k})^k\wedge[C_f]\right)$.
    \item \textbf{Approximation formula}:
    $L(\lambda)=\lim_n d^{-kn}\sum_{z\in{\cali R}_n(\lambda)} \ln \vert {\rm Jac} f_\lambda (z)\vert$.
\end{itemize}~

The first formula has been obtained by Bassanelli and Berteloot in \cite{BB},
it generalizes similar formulas in dimension one due to Przytycki \cite{Prz} and Manning \cite{Man}  for polynomials 
and DeMarco  \cite{DM} for rational functions. The second one was proved by Berteloot, Dupont and Molino in \cite{BDM} and a simplified proof,
avoiding difficulties due to the possible resonances between the Lyapunov exponents, has been given by Berteloot and Dupont in \cite{BDu}.\\


We shall also need the notion of \emph{  acritical equilibrium web} whose existence, as shown in \cite[Theorem 4.1]{BBD},
implies that of an equilibrium lamination and, in principle, is much  easier  to prove.
The projections $p_\lambda : \mathcal O(M,\P^k) \to \P^k $ are defined by $p_\lambda(\gamma):=\gamma(\lambda)$.
\begin{definition}\label{defi web} An \emph{equilibrium web for $f$ on $M$} is a Borel probability measure on  the metric space $(\mathcal O(M,\P^k),d_\luc)$ such that:
    \begin{enumerate}
        \item $\cali{F}_*\cali {M}=\cali {M}$,
        \item $\forall \lambda \in M \: : \: (p_{\lambda})_*\cali{M}=\mu_\lambda$,
        \item $\supp (\cali M)$ is compact.
    \end{enumerate}
Set $\cali J_s := \{\gamma \in \cali J \: : \: \Gamma_\gamma \cap (\cup_{m\ge 0} f^{-m}(\cup_{n\ge 0}f^n(C_f))) \ne \emptyset \}$.
An equilibrium web $\cali M$ is said  \emph{acritical} if $\cali M(\cali J_s) = 0$.
    \end{definition}

In \cite{BBD} an equilibrium web was also required to be supported on $\cali J$, this is actually a consequence of the
above simpler definition. Indeed, if $\gamma_0\notin \cali J$ then there exists $\lambda_0\in M$ for which
$\gamma_0(\lambda_0) \notin \textrm{supp}\; \mu_{\lambda_0}$ and then, taking any neighbourhood $\mathcal V_0$
of $\gamma_0$ in ${\cali O}(M,\P^k)$ such that $p_{\mathcal\lambda_0} (\mathcal V_0)\subset \P^k \setminus  \textrm{supp}\;
\mu_{\lambda_0}$, one gets ${\cali M}(\mathcal V_0) \le {\cali M}(p_{\lambda_0}^{-1}(p_{\lambda_0}(\mathcal V_0))) =
(p_{\lambda_0*}{\cali M) (p_{\mathcal \lambda_0}(\mathcal V_0)) = \mu_{\lambda_0} (p_{\lambda_0}(\mathcal V_0))=0}$.\\

We now show that Theorem \ref{theoStab} can be obtained from the following key lemma, whose proof
 will be given in the next section.
Applying Theorem \ref{theoStab} to the restricted families $f\vert_{\Stab \times \P^k}$
then easily leads to Theorem \ref{MainTheo}.

\begin{lemma}\label{LemmaTech} If $\lambda_0\in \pi_M (\Omega(R_f))$, then there exists a neighbourhood $D_0$ of $\lambda_0$
    in $M$ such that the restricted family $f\vert_{D_0\times \P^k}$ admits an acritical equilibrium web.
\end{lemma}

\proof[Proof of Theorem \ref{theoStab}]

$1) \Leftrightarrow 2).$ This follows immediately from the following estimate which  is a direct consequence of the $dd^c L$-formula
(see \cite[Lemma 3.13]{BBD}).
There exists a positive constant $\alpha$, only depending on $k$ and $m$,
such that
$$\Vert f^n_* [C_f]\Vert_{U\times \P^k}= \alpha \Vert dd^c L\Vert_U + O(d^{(k-1)n})$$
for every relatively compact open subset $U$ of $M$.\\

$2)\Rightarrow 3)$ is obvious.\\

$3) \Rightarrow  4).$ By Lemma \ref{LemmaTech}, any parameter $\lambda_0\in M$ has a neighbourhood $D_0$
such that the restricted family $f\vert_{D_0\times \P^k}$ admits an acritical equilibrium web $\cali M$.
Using Choquet's theory, one shows that $f\vert_{D_0\times \P^k}$ also admits an acritical web which is ergodic
(see \cite[Proposition 2.4]{BBD}).
It then follows from \cite[Theorem 4.1]{BBD} that $f\vert_{D_0\times \P^k}$ admits an equilibrium lamination;
the proof of this theorem fully exploits the ergodicity of the dynamical system $({\cali J},{\cali F},{\cali M})$.
Finally, since $M$ is simply connected, the family
$f$ itself admits an equilibrium lamination $\cali L$ by an analytic continuation argument.\\
 
$4) \Rightarrow  5).$ This has been proved by Bianchi in the wider context of holomorphic families of polynomial-like maps
of large topological degree (see \cite[Theorem 4.11]{Bia}).
The proof consists on a generalization to the dynamical system $({\cali J},{\cali F},{\cali M})$, where $\cali M$
is an ergodic acriticalequilibrium web associated to the lamination $\cali L$, of the strategy developed
by Briend and Duval \cite{BD1} to recover the equidistribution of the repelling cycles from the properties of the equilibrium measure.\\

 $5) \Rightarrow  1).$ For every integer $n$ we have,  by assumption,
a finite collection $(\gamma_{j,n})_{1\le j\le N_n}$ of holomorphic maps
$\gamma_{j,n} : M \to \P^k$ such that $\gamma_{j,n}(\lambda)  \in {\cali R}_n(\lambda)$ for every $\lambda \in M$
and $N_n \sim d^{kn}$. The collection is invariant by the action of $\cali F$. We thus may write
$$d^{-kn}\sum_{z\in{\cali R}_n(\lambda)} \ln \vert {\rm Jac} f_\lambda (z)\vert=
d^{-kn}\sum_{1\le j\le N_n} \ln \vert {\rm Jac} f_\lambda (\gamma_{j,n}(\lambda)\vert +
d^{-kn}\sum_{z\in{\cali R'}_n(\lambda)} \ln \vert {\rm Jac} f_\lambda (z)\vert,$$
where ${\cali R'}_n(\lambda):= {\cali R}_n(\lambda) \setminus \cup_{1\le j\le N_n} \{\gamma_{j,n}(\lambda)\}$.

The second term in the above sum is positive
since the cycles involved in  ${\cali R'}_n(\lambda)$ are repelling, and it is bounded from above
by  $K_\lambda d^{-kn} \vert {\cali R'}_n(\lambda)\vert \le K_\lambda (1-\frac{N_n}{d^{kn}})$
where $K_\lambda$ is a positive constants which depends continuously on $\lambda$. 
Similarly the first term is locally uniformly bounded and, moreover, defines a pluriharmonic function on $M$. 
By the approximation formula, the function $L$  is thus
a limit in $L^1_{loc}(M)$ of pluriharmonic functions.
\qed\\


To end this section we recall the role played by Misiurewicz parameters in the study of stability.
A parameter $\lambda_0\in M$ is called \emph{Misiurewicz} if some repelling periodic point
of $f_{\lambda_0}$ belongs to the post-critical set of $f_{\lambda_0}$ and if this configuration is not stable by small perturbations
(see \cite[Definition 1.5]{BBD} for a precise statement).
Note that Misiurewicz parameters are basic examples of parameters  where, in some sense, the post-critical normality  fails. 

The basic result about Misiurewicz parameters is the following, 
its proof combines the $dd^cL$-formula with an asymptotic  phase-parameter transfer (see \cite[Proposition 3.7]{BBD}).

\begin{proposition}\label{PropMisiu}
 If $dd^c L=0$ on $M$ then  there are no Misiurewicz parameters on $M$.
\end{proposition}
This result has been used by several authors to construct holomorphic
families of endomorphisms for which the bifurcation locus $M\setminus \Stab(f)$ has non empty interior
(see \cite{Bi-Ta}, \cite{Du}, \cite{Ta}, \cite{Bie}),
or for estimating the Hausdorff dimension of slices in the bifurcation locus
 (see \cite{BB2}). Let us also stress that, 
using different techniques than in \cite{BBD}, Bianchi generalized Proposition \ref{PropMisiu} to the setting of holomorphic families of
polynomial-like maps of large topological degree (see \cite[Theorem A]{Bia}).

It has also been shown that Misiurewicz parameters are dense in the bifurcation locus 
(see \cite{BBD} Theorem 1.6) and thus, taking Proposition \ref{PropMisiu} into account,
that the the bifurcation locus is the closure of the subset of Misiurewicz parameters.
This result played an important role in the approach of \cite{BBD} and its proof is quite involved.
The proof of Theorem \ref{theoStab} presented here avoids these difficulties.

\section{Proof of Lemma \ref{LemmaTech}}\label{SectionProof}

We set $Y:=f(C_f)$.
Take $a:=(\lambda_0,z_0) \in \Omega(R_f)$ and let $U=:D\times B$ be a neighbourhood of $a$
in $M\times \P^k$ such that  $\sum_{n\ge 0} \Vert \mathds{1}_U R_{n,f}\Vert < +\infty$. Using Baire's theorem, one sees that  $z_0$ might be slightly moved in $B$ so that
$$ a \notin \cup_{n\ge 0} f^{\circ n} (Y).$$
Since the problem is local, we may shrink $D$ (resp. $B$)  so that it is  contained in a local chart of $M$ (resp. $\P^k$) and therefore
assume that $D\times B \subset \C^{m+k}$.
We then denote by $\cali D$ the set of complex lines in $\C^{m+k}$ passing 
through the point $a$ and, for each $\Delta\in {\cali D}$ and every $\rho>0$ , we denote by $\Delta_\rho$
the euclidean disc lying on $\Delta$, centered at $a$ and of radius $\rho$.  We identify $\cali D$ with
$\P^{m+k-1}$ and endow it with a probability measure $\cali L$ induced by the Fubiny-Study metric
on $\P^{m+k-1}$.\\

From now on, we fix $0< \epsilon < \frac{1}{2}$ and set $\delta := \frac{\epsilon}{1-\epsilon}$.
We will construct an acritical equilibrium web for
$f\vert_{D_0\times \P^k}$ where $D_0$ is some neighbourhood of $\lambda_0$ contained in $D$.
We will proceed in {five} steps. The first three are directly inspired by the work of Dinh and Sibony on the equidistribution of iterated preimages
(see \cite[section 3.4]{DS2} or \cite[Section 1.4]{DS}).\\

{\it Step $1$:  Constructing a large set of discs} on which $f^{\circ n}$ admits almost $d^{kn}$ inverse branches.

\begin{lemma}\label{L1} There exist
$r>0$ and ${\cali D}'\subset {\cali D}$ such that:
\begin{itemize}
\item[i)]  $B(a,r)\cap Y = \emptyset$,
\item[ii)] ${\cali L}({\cali D}') >1-\delta$,
\item[iii)] $\sum_{n=0}^{+\infty} \Vert R_{n,f} \wedge [\Delta_r]\Vert \le \epsilon\;
\forall \Delta\in {\cali D}'$,
\item[iv)] for every integer $n$ and every $\Delta \in {\cali D'}$, there exists a set $\Gamma_r^n(\Delta)$ of
inverses branches of $f^{\circ n}$ above $\Delta_r$ such that $\vert \Gamma_r^n(\Delta)\vert \ge (1-\epsilon)d^{kn}$,
\item[v)]  ${\cali F} (\Gamma_r^{n+1}(\Delta)) \subset \Gamma_r^n(\Delta)$ 
and   $\gamma(\Delta_r) \cap ({ C_f \cup} f(C_f))=\emptyset$ for every integer $n$ and every $\gamma \in \Gamma_r^{n}(\Delta)$.
\end{itemize}
\end{lemma}

\begin{proof}
Take $r>0$ small enough so that the closed euclidean ball  $\overline{B(a,r)}$ in $\C^{m+k}$ is contained in $U$.
{We shall use the following standard fact  (see \cite[Lemma 1.53]{DS}).}

\begin{lemma}\label{1.4.9}
Let $S$ be a positive closed $(1,1)$-current on $U$. Then there exist a family
$\cali D' \subset \cali D$ with $\cali L(\cali D') > 1-\delta$ and a constant $A_\delta > 0$ which is independant of $S$,  such that  the measures $S\wedge [\Delta_r]$ are well-defined and
of mass $\le A_\delta ||S||$ for every $\Delta$ in $\cali D'$.
\end{lemma}

Applying Lemma \ref{1.4.9} for $S=\mathds{1}_U R_f$, we get $A_\delta > 0$ and a family of lines $\cali D' \subset \cali D$
satisfying ii) and such that $\mathds{1}_U R_f \wedge [\Delta_{r}]$ is well defined as a closed positive current with mass less than
$A_\delta\cdot||\mathds{1}_U R_f||$ for any $\Delta \in {\cali D'}$. Removing some $\cali L$-negligible subset  from $\cali D'$ 
allows to assume that the currents $\mathds{1}_U R_{n,f} \wedge [\Delta_{r}]$ are well defined for all $n\in \N$.
This forces the  series $ \displaystyle \sum_{n \in \N} \left\lVert \mathds{1}_U R_{n,f} \wedge [\Delta_{r}] \right\rVert$ to converge
for any line $\Delta \in \cali D'$.  Indeed, for any $N\in \N$ we have:
\begin{equation*}
    \begin{split}
        \sum_{n=0}^N \left\lVert \mathds{1}_U R_{n,f} \wedge [\Delta_{r}] \right\rVert & = \Vert\sum_{n=0}^N  \mathds{1}_U R_{n,f} \wedge [\Delta_{r}] \Vert\\ 
        & \le  ||\mathds{1}_U R_f \wedge [\Delta_{r}]|| \le A_\delta\cdot||\mathds{1}_U R_f||  < +\infty.
    \end{split}
\end{equation*}

Let us set $\cali D'(N) := \{\Delta \in \cali D' \::\: N(\Delta) \leq N \}$, where $N(\Delta)$ is the smallest positive integer for which
$ \sum_{n \ge N(\Delta)} \left\lVert \mathds{1}_U R_{n,f} \wedge [\Delta_{r}] \right\rVert \le \epsilon$.
As the union $\cali D' = \cup_{N\in \N} \cali D'(N)$ is increasing,
we may assume, after replacing ${\cali D'}$ by $\cali D'(N)$ for $N$ big enough, that 
\begin{equation}\label{queue}
\forall \Delta\in {\cali D'}: \: \sum_{n \ge N} \left\lVert \mathds{1}_U R_{n,f} \wedge [\Delta_{r}] \right\rVert \le \epsilon.
\end{equation}

Since $a\notin \cup_{n\le N} f^{\circ n}(Y)$, we may reduce $r$ so that $$B(a,r) \cap \cup_{n\le N} f^{\circ n}(Y) = \emptyset.$$
It follows that  i) is satisfied and that $R_{n,f} \wedge [\Delta_r] = 0$ for any $n \le N$.
Combining this with (\ref{queue}) yields iii).

Let us establish  iv). Set $\epsilon_n^\Delta := ||\mathds \mathds{1}_U R_{n,f} \wedge [\Delta_r]||$.
Note that  $d^{kn} \epsilon_n^\Delta$ is the cardinality of the  intersection of $f^n(Y)$ with $\Delta_r$, counting multiplicities and that, in particular, $\epsilon_n^\Delta=0$ for $n\le N$.
Given $\Delta \in \cali D'$, we shall prove by induction on $n$ that $f^{\circ n}$ admits at least $\nu_n^\Delta$ 
inverse branches which do  not meet $Y$ on $\Delta_r$ where
\begin{equation}\label{nun}
\nu_n^\Delta \ge (1 - \sum_{i=0}^n \epsilon_i^\Delta ) d^{kn}.
\end{equation}

The base case is covered by i). For the induction step, assume the hypothesis for $f^{\circ n}$.
As $f$ realizes an unramified covering of degree $d^k$, $f^{n+1}$ admits at least 
$\left(1 - \sum_{i=0}^n \epsilon_i^\Delta \right) d^{k(n+1)}$ inverse branches above $\Delta_r$.
Among them, at most $d^{k(n+1)} \epsilon_{n+1}^\Delta$ do intersect $Y$, which leads to the desired property for $f^{\circ(n+1)}$.
Now iii) and (\ref{nun}) immediately yield the announced  estimate $\nu_n^\Delta \ge (1-\epsilon)d^{kn}$.

The assertion  v) directly follows from the above construction.
\end{proof}

{\it Step $2$: Estimating the number of points in $(f^{\circ n})^{-1}(a)$ which belong to a lot of inverse branches
of the form $(f^{\circ n})^{-1}(\Delta_r)$ given by step 1.}

\begin{lemma}\label{L2} Let $r>0$ be given by Lemma \ref{L1}. Set $(f^{\circ n})^{-1}(a)=:\{a_1^n,\cdots,a_{l_n}^n\}$ and 
let $I_n:{\cali D}'\to {\cali P}\{1,2,\cdots, l_n\}$ be the map which associates to any $\Delta\in {\cali D'}$
the subset of $s\in \{1,2,\cdots, l_n\}$ such that there exist an inverse branch of  $f^{\circ n}$ defined on  $\Delta_r$ and passing through $a_s^n$.

Set ${\cali D}_s^{r,n}:= \{\Delta \in \cali D' \: : \: s\in I_n(\Delta) \}$ and 
${\cali S}_{\epsilon,r,n}:=\{1\le s\le l_n\;\colon\; {\cali L}({\cali D}_s^{r,n}) \ge 1-2\sqrt{\epsilon}\}$, for  $1\le s\le l_n$.
Then:
\begin{itemize}
\item[i)] $(1-\sqrt{\epsilon})d^{kn} \le \vert {\cali S}_{\epsilon,r,n}\vert \le d^{kn}$;
\item[ii)] there exists  $i_n : \cali S_{\epsilon, r, n} \to \cali S_{\epsilon, r, n-1}$ such that
$f(a_s^n) = a_{i_n(s)}^{n-1}$ and  $|i_n^{-1}(\{s\})| \le d^k$, for any $s\in \cali S_{\epsilon, r, n}$;
\item[iii)] the sequence $(d^{-kn}|\cali S_{\epsilon, r, n}|)_n$ 
converges.
\end{itemize} 
\end{lemma}

\begin{proof}

We first establish the following estimate:
\begin{eqnarray}\label{estimateFsn}
\sum_{s=1}^{l_n} {\cali L} ({\cali D}_s^{r,n}) \ge (1-\delta)(1-\epsilon)d^{kn}.
\end{eqnarray}

Let $\nu$ be the counting measure on $\{1,2,\dots,l_n\}$ and set $ \widetilde{\cali L} := \nu \otimes \cali L$.
Consider the following subset $Q_n := \{(s,\Delta) \::\: \{s\} \in I_n(\Delta)\}$ of $\{1,2,\dots,l_n\} \times \cali D'$.
We compute $\widetilde{\cali L}(Q_n)$ with two different partitions of $Q_n$.
The first one  is based on the value of $s \in \{1,2,\dots,l_n\}$:
\begin{equation}\label{counting}
\widetilde{\cali L}(Q_n) = \widetilde{\cali L}\left( \bigcup_{s=1}^{l_n} \{s\} \times {\cali D}_s^{r,n}\right)
        = \sum_{s=1}^{l_n} \widetilde{\cali L} \left( \{s\} \times {\cali D}_s^{r,n}\right) = \sum_{s=1}^{l_n} \cali L ( {\cali D}_s^{r,n} ).
\end{equation}
The second one  is based on the value of $I_n(\Delta) \in {\cali P}\{1,2,\cdots, l_n\}$:
$$
\widetilde{\cali L}(Q_n) = \widetilde{\cali L}\left( \bigcup_{\cali S \in \cali P \{1,2,\dots,l_n\}} \cali S \times I_n^{-1}(\{\cali S\}) \right)
= \sum_{\cali S \in \cali P \{1,2,\dots,l_n\}} \nu(\cali S) \cdot \cali L(I_n^{-1}(\{\cali S\})).
$$
As soon as $|\cali S| < (1-\epsilon)d^{kn}$, we get from Lemma \ref{L1} iv) that $I_n^{-1}(\{\cali S\}) = \emptyset$. Then the last equality leads to 
\begin{eqnarray}\label{Counting}
\widetilde{\cali L}(Q_n) \geq (1-\epsilon)d^{kn} \cali L(\cali D'),
\end{eqnarray}
and \eqref{estimateFsn} immediately follows from \eqref{counting}, \eqref{Counting} and Lemma \ref{L1} ii).
\\
We can now end the proof of the Lemma.
Since $l_n \le d^{kn}$, in order to prove i), it is sufficient to prove the first inequality,
namely that $(1-\sqrt{\epsilon})d^{kn} \leq |\cali S_{\epsilon,r,n}|$.
One has
\begin{equation*}
    \begin{split}
        \sum_{s=1}^{l_n} \cali L(\cali D_s^{r,n}) & = \sum_{s\in \cali S_{\epsilon,r,n}} \cali L(\cali D_s^{r,n}) + \sum_{s\notin \cali S_{\epsilon,r,n}}\cali L(\cali D_s^{r,n})\\
        & \le |\cali S_{\epsilon, r, n}| + (l_n - |\cali S_{\epsilon, r, n}|)(1-2\sqrt{\epsilon}) \\
        & \le 2\sqrt{\epsilon} |\cali S_{\epsilon, r, n}| + (1-2\sqrt{\epsilon})d^{kn},
    \end{split}
\end{equation*}
which, combined with \eqref{estimateFsn}, gives that
$2\sqrt{\epsilon} |\cali S_{\epsilon, r, n}| \ge \left[ (1-\delta)(1-\epsilon) - (1-2\sqrt{\epsilon})\right] d^{kn}$.
Then, by our  choice of $\delta$, we get:
$$|\cali S_{\epsilon, r, n}| \ge \left[ 1- \frac{\epsilon + \delta(1-\epsilon)}{2\sqrt{\epsilon}}\right] d^{kn} = (1-\sqrt{\epsilon})d^{kn},$$
which gives i).

Let us finally justify ii) and iii).
We define the map $i_n$ on $\cali S_{\epsilon, r, n}$ by $f(a_s^n) = a_{i_n(s)}^{n-1}$. It then follows from the assertion v) of Lemma \ref{L1} and the very definition of   $\cali S_{\epsilon, r, n}$ that $i_n( \cali S_{\epsilon, r, n})\subset  \cali S_{\epsilon, r, n-1}$.  Since $f$ is a ramified covering of degree $d^k$, one has $|i_n^{-1}(\{s\})| \le d^k$.

Since $|\cali S_{\epsilon, r, n}| = \sum_{s\in \cali S_{\epsilon,r,n-1}} |i_n^{-1}(s)| \le d^k |\cali S_{\epsilon,r,n-1}|$, the sequence
$(d^{-kn}|\cali S_{\epsilon, r, n}|)_n$ is positive non-increasing and thus converges.
\end{proof}

{\it Step $3:$ Extension of the inverse branches to a ball $B(a,r)$.}\\

We will prove here the existence
of at least $(1-\sqrt{\epsilon})d^{kn}$  inverse branches of $f^{\circ n}$ on an open  ball  $B(a,\tau r) \subset B(a,r)$ centered at $a$
in $\C^{m+k}$. To achieve this, we will  combine  Lemma \ref{L2} with the extension theorem of Sibony-Wong
(see \cite{SW} or \cite[Theorem 1.54]{DS}) which we recall below.

\begin{theo}{\bf (Sibony-Wong)}\label{Sibony-Wong}
 Let $\cali D'' \subset \cali D$ be such that $\cali L(\cali D'')\ge c$ for a positive constant $c$ and
    let $\Sigma$ denote the intersection of  $\cali D''$ with $B(a,r)$.
    Then there exists $\tau \in (0,1)$, independant from $r$ and $\cali D'$, such that any holomorphic function $h$ on a neighbourhood of $\Sigma$ can be extended
    to a holomorphic function $\tilde h$ on $B(a, \tau r)$.
    Moreover, the extended function $\tilde h$ enjoys the following estimate: $\sup_{B(a,\tau r)} \vert \tilde h\vert \leq \sup_\Sigma \vert h\vert$.
\end{theo}

Our precise statement is as follows; we keep  the notations introduced in Lemma \ref {L2}. 

\begin{lemma}\label{L3}
There exists $\tau\in ]0,1[$ such that  to every $s\in {\cali S}_{\epsilon,r,n}$ is associated an  inverse branch
$\gamma_s^n : B(a,\tau r) \to \C^m\times \P^k$ of $f^{\circ n}$ such that $\gamma_s^n(a)=a_s^n$. Moreover,
 $f \circ \gamma_s^n = \gamma_{i_n(s)}^{n-1}  \text{ for every } s\in \cali S_{\epsilon, r, n}$.
\end{lemma}

\begin{proof}
Fix $n\in \N$ and $s\in \cali S_{\epsilon, r, n}$.
For every  $\Delta \in {\cali D}_s^{r,n} = \{\Delta \in \cali D' \: : \: s \in I_n(\Delta)\})$,
we  denote by $\gamma_{s,\Delta}^n$ the inverse branch of $f^{\circ n}$ above $\Delta_r$ such that $\gamma_{s,\Delta}^n(a) = a_s^n$.
Observe that the holomorphic maps $\gamma_{s,\Delta}^n$ and $\gamma_{s,\Delta'}^n$
are actually defined on neighbourhoods of $\Delta_r$ and $\Delta'_r$,
and coincide on some neighbourhood of $a$, for any pair of lines $\Delta$, $\Delta' \in \cali D_s^{r,n}$.
By analytic continuation, we can therefore consider the branch $\gamma_s^n$ as defined on a neighbourhood of $\cup_{\Delta \in \cali D_s^{r,n}} \Delta_r$.
    
By definition of   $\cali S_{\epsilon, r, n}$ (see Lemma \ref{L2}),
one has $\cali L(\cali D_s^{r,n}) \ge 1- 2\sqrt{\epsilon}$ and we may therefore apply
Theorem \ref{Sibony-Wong} with $c = 1-2\sqrt{\epsilon}$ and
$\cali D'' = \cali D_s^{r,n}$ to the coordinate functions of $\gamma_s^n$.
In this way, one sees that each $\gamma_s^n$ extends to an inverse branch of $f^{\circ n}$ defined on $B(a,\tau r)$.

 When  $s\in \cali S_{\epsilon, r, n}$, the map  $f \circ \gamma_s^n$ is an inverse branch of $f^{\circ(n-1)}$ which is defined on $B(a,\tau r)$ and 
    whose value at $a$ is  $f(a_s^n) = a_{i_n(s)}^{n-1}$. It must therefore coincide with the branch $\gamma_{i_n(s)}^{n-1}$ which,
by  assertion ii) of Lemma \ref{L2} and the above construction, does exist.
\end{proof}

The following lemma is also a consequence of the Sibony-Wong extension theorem, we shall use it in the fifth (and last) step.
\\
\begin{lemma}\label{CSW} Fix $\rho > 0$. For $u \in {\cali O}(B(a,\rho))$ we set
$\cali{D}_u:=\{\Delta \in {\cali D} \;\colon\; 0 \notin u (\Delta_\rho)\}$.
There exists $0<\tau'<1$ such that  $0 \notin u(B(a,\tau' \rho))$  as soon as ${\cali L}(\cali{D}_u) >\frac{1}{2}$.
\end{lemma}

\proof If  $u \in {\cali O}(B(a,\rho))$ and ${\cali L}(\cali{D}_u) >\frac{1}{2}$
then $u(a)\ne 0$ and thus, by Hurwitz lemma, the function
 $u$ does not vanish on $\Delta_{\rho}$ for every $\Delta \in \overline{\cali{D}_u}$.
This shows that $\cali D_u$ is closed and that the function
$h:=\frac{1}{u}$ is holomorphic on some neighbourhood of $\Sigma_u:= {\cali D}_u \cap B(a,\rho)$.
By Theorem \ref{Sibony-Wong}, this function extends to $\tilde {h}  \in  {\cali O}(B(a,\tau'\rho))$
where  $0<\tau' <1$ neither depends on $u$ or $\rho$.
Since $u\tilde {h} =uh=1$ on $\cup_{\Delta \in {\cali D_u}} \Delta_{\tau' \rho}$
which has positive Lebesgue measure, one has $u\tilde {h}=1$, and therefore $u$  does not vanish on $B(a,\tau' \rho)$.\qed\\

{\it Step $4:$ Construction of an equilibrium web.}\\

We will use the collection of inverse branches $(\gamma_s^n)_{n\ge 1,  s\in {\cali S}_{\epsilon,r,n}}$
obtained in the former step to build an equilibrium web $\cali M$
for the restricted family $f\lvert_{D_0\times \P^k}$ and some neighbourhood $D_0$ of $\lambda_0$.
The equidistribution of iterated preimages towards the measures $\mu_\lambda$,
as well as the ergodicity of these measures, will play an important role here.\\

Let  $D_0 \times B_0$ be a neighbourhood of $a$ contained in $B(a,\tau r)$, where $\tau >0$ is given by Lemma \ref{L3}. For every $z\in B_0$, 
we define  a sequence of discrete measures $(m_n(z))_n$ on $\mathcal O (D_0,\P^k)$, and  their Cesàro means $({\cali M}^n(z))_n$, by  
$$m^n(z):= d^{-kn} \sum_{s\in \cali S_{\epsilon, r,n}} \delta _{\gamma_{s,z}^n} \; \quad \textrm{and} \quad \;
\cali M^n(z) := \frac{1}{n} \sum_{r=1}^n m^r(z),$$
where $\gamma_s^n(\lambda, z) =: (\lambda, \gamma_{s,z}^n (\lambda))$.
If $\vol$ denotes the Lebesgue measure on $B_0$, we then set
$$ m^n := \frac{1}{\vol(B_0)}\int_{B_0}  m^n(z)\;d\vol(z)\;\textrm{and}\;
\cali M^n :=\frac{1}{\vol(B_0)}\int_{B_0}  {\cali M}^n(z)\;d\vol(z)= \frac{1}{n} \sum_{r=1}^n m^r.$$

The equilibrium web $\cali M$ will be obtained as a rescaled weak limit of $({\cali M}^n)_n$. In the above definitions,
the averaging on $B_0$ is devoted to make $z$ avoid the exceptional set of $f_\lambda$ for Lebesgue-almost every $z$
and every fixed $\lambda$. The Cesàro means will allow us to get the $\cali F$-invariance of $\cali M$.
The compactness of the support of $\cali M$ will be obtained from the following special case of  a classical result of Ueda
(see \cite[Theorem 2.1]{UED}).

 \begin{lemma}\label{remU} The family  of all inverse branches of $f^{\circ n}$ on $B(a,\tau r)$, when $n$ runs over $\N$,
 is a normal family in  $\mathcal O(B(a,\tau r), \P^k)$. In particular, for any  $\alpha>0$, we may reduce $r$ so that $d_{\P^k} (\gamma_s^n(z), \gamma_s^n(a)) \le \alpha
 $ for all points $z\in B(a,\tau r)$ and all branches $\gamma_s^n$ given by Lemma \ref{L3}.
\end{lemma}

We may now state the main result of this step.

\begin{lemma}\label{L4} Let  $D_0 \times B_0$ be any neighbourhood of $a$ contained in $B(a,\tau r)$.
    There exists a positive measure $\widetilde{\cali M}$ on $\mathcal O (D_0,\P^k)$
    such that a subsequence of  $(\cali M^n)_n$ converges to $\widetilde{\cali M}$
    and the probability measure $\displaystyle \cali M:=\frac{\widetilde{\cali M}}{||\widetilde {\cali M}||}$
    is an equilibrium web for the restricted family 
    $f\vert_{D_0\times \P^k}$.
\end{lemma}

\begin{proof}
 By definition we have
\begin{equation*}
    \begin{split}
        ||\cali M^n||  = \frac{1}{n} \sum_{r=1}^n ||m^r|| = \frac{1}{n} \sum_{r=1}^n \frac{1}{\vol(B_0)} \int_{B_0} ||m^n(z)|| d\vol(z)
         = d^{-kn} |\cali S_{\epsilon,r,n}|,
    \end{split}
\end{equation*}
which, by the last assertion of Lemma \ref{L2}, yields the existence of $\alpha\in [0,1[$ such that 
$$\lim_n ||\cali M^n|| =1-\alpha.$$

Since the family  of all inverse branches of $f^{\circ n}$ above $B(a,\tau r)$, when $n\in \N$,
is a normal family in  $\mathcal O(B(a,\tau r), \P^k)$,  due to Lemma \ref{remU}, the family of their restrictions to $D_0\times \{z\}$ when $z\in B_0$ is 
normal as well. In other words, there exists a compact subset $\mathcal K$ of the space $(\mathcal O(B(a,\tau r), \P^k), d_\luc)$ such that
$\supp (\cali M^n) \subset \mathcal K$ for all integers $n$.
 By Banach-Alaoglu theorem, the sequence $(\cali M^n)_n$ admits a cluster value $\widetilde{\cali M}$ with support in $\mathcal K$.\\

Let us now show that $\widetilde{\cali M}$ is $\cali F$-invariant.
To this purpose, we first establish the following estimate:
\begin{equation}\label{diff}
a_n(z) := ||m^{n-1}(z) - \cali F_*m^n(z)||  \le
\frac{|\cali S_{\epsilon, r ,n-1}|}{d^{k(n-1)}} - \frac{|\cali S_{\epsilon, r ,n}|}{d^{kn}} =:a_n,
\; \forall z\in B_0.
\end{equation}
We observe that  $\cali F(\gamma_{s,z}^n) = \gamma_{i_n(s),z}^{n-1}$. Indeed, by Lemma \ref{L3} we have, for every $\lambda \in D_0$:
$$ \cali F(\gamma_{s,z}^n)(\lambda) =
f_\lambda(\gamma_s^n(\lambda,z)) = \Pi_{\P^k} \circ f \circ \gamma_s^n(\lambda,z) =
\Pi_{\P^k} \circ \gamma_{i_n(s)}^{n-1}(\lambda, z) = \gamma_{i_n(s),z}^{n-1}(\lambda),$$
where $\Pi_{\P^k} : M \times \P^k \to \P^k$ is the canonical projection.
Hence, by partitioning $\cali S_{\epsilon,r,n}$ on the values of $i_n(s)$ , where the function $i_n$ is defined in Lemma \ref{L2}, one gets

\begin{equation*}
    \begin{split}
        \cali F_* m^n(z) & = d^{-kn} \sum_{s\in \cali S_{\epsilon, r,n}} \cali F_* (\delta_{\gamma_{s,z}^n}) =
        d^{-kn} \sum_{s\in \cali S_{\epsilon, r,n}} \mathcal \delta_{\gamma_{i_n(s),z}^{n-1}}
        = d^{-kn} \sum_{s\in \cali S_{\epsilon, r, {n-1}}} |i_n^{-1}(s)| \delta_{\gamma_{s,z}^{n-1}},
    \end{split}
\end{equation*}
and thus $m^{n-1}(z) - \cali F_*m^n(z)  = d^{-kn} \sum_{s\in \cali S_{\epsilon,r,{n-1}}}(d^k - |i_n^{-1}(s)|) \delta_{\gamma_{s,z}^{n-1}}$. Then, 
as  $|i_n^{-1}(s)| \le d^k$ (see Lemma \ref{L2}), we obtain 

\begin{equation*}
        a_n(z)  \le d^{-kn} \sum_{s\in \cali S_{\epsilon,r,{n-1}}} d^k - |i_n^{-1}(s)| 
         = d^{-k(n-1)}|\cali S_{\epsilon,r,n-1}| - d^{-kn} \sum_{s\in \cali S_{\epsilon,r,{n-1}}}|i_n^{-1}(s)|,
\end{equation*}
which  is the desired inequality, since $\displaystyle \sum_{s\in \cali S_{\epsilon,r,{n-1}}}|i_n^{-1}(s)| = |\cali S_{\epsilon,r,n}|$.\\
It immediately follows from  (\ref{diff}) that
\begin{eqnarray}\label{invar}
||m^{r-1} - \cali F_*m^r|| \le \frac{1}{\vol(B_0)} \int_{B_0}a_r(z) d\vol(z) \le a_r.
\end{eqnarray}

Now, since
$$
    \cali M^n - \cali F_* \cali M^n
    = \frac{1}{n} \sum_{r=1}^n m^r - \frac{1}{n} \sum_{r=1}^n \cali F_* m^r
    = \frac{1}{n} \sum_{r=2}^n (m^{r-1} - \cali F_* m^r) + \frac{m^n-\cali F_* m^1}{n},$$
one deduces from (\ref{invar}) that
\begin{equation*}
    \begin{split}
        ||\cali M^n - \cali F_* \cali M^n|| & \le \frac{1}{n} \sum_{r=2}^n ||m^{r-1} - \cali F_*m^r|| + \frac{2}{n}
        \le \frac{1}{n} \sum_{r=2}^n a_r + \frac{2}{n} \\
        & \le \frac{1}{n} \left( \frac{|\cali S_{\epsilon, r, 1}|}{d^k} - \frac{|\cali S_{\epsilon, r, n}|}{d^{kn}} \right) + \frac{2}{n}
        \le \frac{3}{n}.
    \end{split}
\end{equation*}
Since $\cali F_*$ is continuous for the weak topology,
this proves that $\cali F_* \widetilde{\cali M }= \widetilde{\cali M }$.\\

Now $\cali M:=\frac{\widetilde{\cali M}}{1-\alpha}$ is a compactly supported $\cali F$-invariant probability measure
on  the metric space $(\mathcal O(B(a,\tau r), \P^k), d_\luc)$,
and it remains to prove that  $(p_\lambda)_*(\cali M) = \mu_\lambda$  for any $\lambda \in D_0$.
For $(\lambda,z) \in D_0 \times B_0$ and $n\in \N$, let us set $\mu_\lambda^n(z) := d^{-kn} (f_\lambda^{\circ n})^*\delta_z$.
Then, by definition, we have $0\le (p_\lambda)_*(\cali M^n(z)) \le \frac{1}{n}\sum_{r=1}^n \mu_\lambda^r(z)$ for every $z\in B_0$ and therefore:
\begin{eqnarray}\label{EquiDis}
0\le (p_\lambda)_*(\cali M^n) \le \frac{1}{\vol(B_0)} \int_{B_0} \frac{1}{n} \sum_{r=1}^n \mu_\lambda^r(z)\; d\vol(z).
\end{eqnarray}
As the exceptional set of $f_\lambda$ is a proper pluripolar subset of $\P^k$ \cite{FS2},
it follows from the equidistribution theorem \cite{DS} that 
$(\mu_\lambda^n(z))_n$ is weakly converging to $\mu_\lambda$ for almost every $z\in B_0$.
Then, by  Lebesgue convergence Theorem, the right hand side of (\ref{L1}) is weakly converging to $\mu_\lambda$ and thus
$$ p_{\lambda*}{\widetilde {\cali M}} \le \mu_\lambda,\; \forall \lambda \in D_0.$$
Since $\Vert  p_{\lambda*}{\widetilde {\cali M}}\Vert =\Vert {\widetilde {\cali M}} \Vert = 1-\alpha$
where $\alpha\in [0,1[$ we are done if $\alpha=0$. Otherwise this allows to write
$\mu_\lambda = (1-\alpha)\frac{(p_\lambda)_*(\widetilde{\cali M})}{1-\alpha} +
\alpha \frac{\mu_\lambda - (p_\lambda)_*(\widetilde{\cali M})}{\alpha}$
which, by the ergodicity of $\mu_\lambda$, implies that
$\frac{(p_{\lambda*}) (\widetilde{\cali M})}{1-\alpha}= p_{\lambda *}({\cali M})=\mu_\lambda$, as well.
\end{proof}

{\it Step $5:$ The equilibrium web $\cali M$ is acritical.}\\

We first show that we can reduce $r$ and the neighbourhood $D_0\times B_0$ of $a$ so that
$D_0\times B_0\subset B(a, \tau'\tau r)$, where $0<\tau'<1$ is given by
Lemma \ref{CSW}, and, moreover, for every $p\ge 0$ and every branch $\gamma_s^n$ used in the definition of $\cali M$,
there exists a single holomorphic function  defining $f^{\circ p}(Y)$
on some open set containing  $ \gamma_s^n(B(a,\tau r))$.

To do that, 
 we first  reduce $r$ so  that $\gamma_s^n(B(a,\tau r))$ is contained in the domain of definition of some holomorphic section of the
canonical projection $\pi: \C^{k+1}\setminus\{0\} \to \P^k$
for any  $\gamma_s^n$. It suffices for that to use Lemma \ref{remU} with $\alpha>0$ small enough.
Then we shrink $D_0\times B_0$ so that $D_0\times B_0\subset B(a, \tau'\tau r)$, 
and that the family $f\vert_{D_0\times \P^k}$ can be lifted to some holomorphic family 
$F\vert_{D_0\times \C^{k+1}}$ of non-degenerate
homogeneous polynomials of $\C^{k+1}$.
Let $C_F$ be the critical set of $F$.
As the maps $F^{\circ p} : D_0\times \C^{k+1}\to  D_0\times \C^{k+1}$ are proper,
the function $\varphi_p(z) := \Pi_{i=1}^{d^k} \Jac ~F(\rho_i(z))$,
where $\{\rho_1(z),\dots,\rho_{d^k}(z)\}$ is the set of preimages of $z$ counted with multiplicities,
is a holomorphic defining function for the analytic hypersurface $F^{\circ p}(C_F)$ on $D_0 \times \C^{k+1}$.
It follows that for any open set $U$ in $\P^k$ on which $\pi: \C^{k+1}\setminus\{0\} \to \P^k$ admits a holomorphic section,
each analytic set $f^{\circ p}(C_f)$ is defined by a single holomorphic function on $D_0\times U$.\\

The following fact will be crucial.
 Recall that $\cali D$ is the set of one-dimensional lines $\Delta \subset \C^{m+k}$ passing through $a$.

\begin{lemma}\label{LemmaSeries}
If $B(a,\rho)\subset \Omega(R_f)$ then the series $\sum_{l\ge 0}  \int_{\cali D}\Vert \mathds{1}_{B(a, \rho)} R_{l,f} \wedge [\Delta] \Vert$ is converging.
\end{lemma}

\proof Recall the Crofton formula $\int_{\cali D} [\Delta] = (dd^c \log \vert z-a\vert)^{m+k-1}$
(see for example \cite[Corollary III.7.11]{DEM}).
In particular, since the $p.s.h$  function $\log \vert z-a\vert$ is locally bounded on
$\C^{m+k}\setminus\{a\}$,
the currents $u\int_{\cali D} [\Delta]$ and
$dd^c u\wedge\int_{\cali D} [\Delta]$ are well defined
for any $p.s.h$ function $u$  (see \cite[Proposition III.4.1]{DEM}).
Taking for $u$ any local potential of $R_{l,f}$ and exploiting the fact that $u$ is bounded in a neighborhood of $a$,
one may check that $R_{l,f} \wedge[\Delta]$ is $\cali L$-integrable, and that
$\int_{\cali D} R_{l,f} \wedge[\Delta] = R_{l,f}\wedge \int_{\cali D} [\Delta]$.
We then have for any $N \in  \N$:
\begin{eqnarray*}
&\;&\sum_{l=0}^N \int_{\cali D}  \Vert \mathds{1}_{B(a, \rho)} R_{l,f} \wedge [\Delta] \Vert
= \sum_{l=0}^N \Vert  \int_{\cali D} \mathds{1}_{B(a, \rho)} R_{l,f} \wedge  [\Delta] \Vert
= \sum_{l=0}^N \Vert \mathds{1}_{B(a, \rho)} R_{l,f} \wedge  \int_{\cali D} [\Delta] \Vert \\
&=&  \Vert ( \sum_{l=0}^N \mathds{1}_{B(a, \rho)} R_{l,f}) \wedge  \int_{\cali D} [\Delta] \Vert 
\le  \Vert ( \sum_{l\ge 0} \mathds{1}_{B(a, \rho)} R_{l,f}) \wedge  \int_{\cali D} [\Delta] \Vert 
=  \Vert  \mathds{1}_{B(a, \rho)} R_{f} \wedge  \int_{\cali D} [\Delta] \Vert .
\end{eqnarray*}
The assertion follows.
 \qed\\

We have to show that ${\cali M}(\cali J_s)=0$ where $\cali J_s$ is the set of $\gamma \in \cali J$  whose graphs $\Gamma_\gamma$ do  meet the
grand orbit of $C_f$, or equivalently $Y$, by $f$ (see definition \ref{defi web}). The following lemma reduces the problem to estimating the mass of some open subsets of $(\cali J, d_{\textrm{luc}})$.
 For any integer $p$, we define $$Y_p:=\{\gamma \in {\cali J}\;\colon\; \Gamma_\gamma \cap f^{\circ p} (Y) \ne \emptyset\;\textrm{and}\;  \Gamma_\gamma \not \subset f^{\circ p} (Y)\}.$$

\begin{lemma}\label{lemmaYp}
 If $\cali M (Y_p)=0$ for all $p$, then ${\cali M}(\cali J_s)=0$.
\end{lemma}

\proof Note that $\cali J_s=\cup_{n\in \N} {\cali J}_s^n$,
where $ {\cali J}_s^n:=\{\gamma \in {\cali J}\;\colon\; \Gamma_\gamma \cap (f^{\circ n})^{-1}(\cup_{p\in \N}f^{\circ p} (Y)) \ne \emptyset\}$
and that ${\cali J}_s ^0 = \cup _{p\in \N} \widetilde{Y_p}$ where $ \widetilde{Y_p}:=\{ \gamma \in \cali J\;\colon\; 
\Gamma_\gamma \cap f^{\circ p} (Y)\ne \emptyset\}$.
As the measure $\cali M$ is $\cali F$-invariant,
the conclusion follows immediately from the inclusion $ {\cali J}_s^n \subset ({\cali F}^{\circ n})^{-1} ({\cali J}_s^0)$
and the fact that ${\cali M}( \widetilde{Y_p} \setminus Y_p)=0$ (see \cite[end of proof of Corollary 1.7]{BBD}).\qed\\

Let $p\ge 0$ be a fixed integer. It follows from Hurwitz lemma that $Y_p$ is an open subset of $({\cali O}(D_0, \P^k), d_{\textrm{luc}})$. We will show 
that $\lim_n {\cali M}^n (Y_p)=0$.

By the very definition of $\cali M^n$ and since $D_0\times B_0 \subset B(a, \tau'\tau r)$, we have:
\begin{eqnarray}\label{MnYp1}
\cali M^n(Y_p)\le d^{-kn} \;|\{s\in{\cali S}_{\epsilon,r,n}\;\colon\; \gamma_s^n(B(a,\tau'\tau r))\cap f^{\circ p}(Y)\ne \emptyset\}|.
\end{eqnarray} 
Let $\varphi_{p,n}$ be a holomorphic function defining $f^{\circ p}(Y)$ on $ \gamma_s^n(B(a,\tau r))$ and set
\begin{eqnarray*}
    \{s_1,s_2,\cdots,s_{N_n}\} := \{s\in{\cali S}_{\epsilon,r,n}\;\colon\; 0\in \varphi_{p,n} \circ \gamma_s^n (B(a,\tau'\tau r))\}.
\end{eqnarray*}

Then the estimate (\ref{MnYp1}) can be written as 
\begin{eqnarray}\label{MnYp2}
\cali M^n(Y_p)\le d^{-kn} N_n.
\end{eqnarray} 
For every $1\le k \le N_n$, let us set
\begin{eqnarray*}
E_k:=\{\Delta \in \cali D\;\colon\; 0 \in \varphi_{p,n} \circ \gamma_{s_k}^n (\Delta_{\tau r})\}.
\end{eqnarray*}
Then $\sum_{k=1}^{N_n} \mathds{1}_{E_k}(\Delta)$ is the number of inverse branches of $f^{\circ n}$ whose restrictions to $\Delta_{\tau r}$
meet $f^{\circ p}(Y)$ and thus, for a generic $\Delta$ in $\cali D$, we have
\begin{eqnarray}\label{MnYp3}
\sum_{k=1}^{N_n} \mathds{1}_{E_k}(\Delta) \le \vert f^{\circ (n+p)}(Y) \cap \Delta_{\tau r}\vert \le d^{k(n+p)} \Vert R_{n+p,f} \wedge [\Delta_{\tau r}]\Vert.
\end{eqnarray} 
On the other hand, since $\varphi_{p,n} \circ \gamma_s^{n_k}$ vanishes on $B(a, \tau'\tau r)$ for every $1\le k\le N_n$, it follows from Lemma \ref{CSW} that
$\cali L(E_k) >\frac{1}{2}$ which yields:

\begin{eqnarray}\label{MnYp4}
\frac{N_n}{2} \le  \sum_{k=1}^{N_n} \cali L (E_k) = \int_{\cali D} \sum_{k=1}^{N_n} \mathds{1}_{E_k}(\Delta).
\end{eqnarray}

By combining (\ref{MnYp2}), (\ref{MnYp3}), and (\ref{MnYp4}), we obtain that
$$\frac{1}{2} \cali M ^n (Y_p)  \le d^{kp} \int_{\cali D} \Vert R_{n+p,f} \wedge [\Delta_{\tau r}]\Vert $$
which, according to Lemma \ref{LemmaSeries}, implies that $\lim_n \cali M^n (Y_p) =0$.

Since ${\cali M}^n$ is compactly supported in  $({\cali O}(D_0, \P^k), d_{\textrm{luc}})$ and is a weak limit of the sequence
$(\frac{1}{1-\alpha}{\cali M}^n)_n$ (see Lemma \ref{L4}), it follows that $\cali M (Y_p)=0$. By Lemma \ref{lemmaYp}, the equilibrium web
$\cali M$ is acritical. \qed\\

\printbibliography

\end{document}